\documentclass[a4paper,12pt]{amsart}
\usepackage[letterpaper, margin=1.2in]{geometry}
\usepackage{latexsym}
\usepackage{amssymb}
\usepackage{amsmath}
\usepackage{amsfonts}
\usepackage{color}

\newcommand{\Z}{\mathbb Z}

\newcommand{\p}{\mathfrak{p}}

\newcommand{\al}{{\alpha}}

\newtheorem{thm}{\bf Theorem}[section]

\newtheorem{lem}{\bf Lemma}[section]

\newtheorem{prop}{\bf Proposition}[section]

\def\ol{\overline}

\par\vskip 1cm

\newcommand{\biindice}[3]%
{

\begin{array}[t]{c}
#1\\
{\scriptstyle #2}\\
{\scriptstyle #3}
\end{array}
}

\begin{document}

\title[]{Polynomial Index in Dedekind Rings}



\address[M.E. Charkani]{Department of Mathematics, Faculty of Sciences Dhar-Mahraz,  B.P. 1796, Fes, Morocco} \email{mcharkani@gmail.com}
\address[A. Deajim]{Department of Mathematics, King Khalid University,
P.O. Box 9004, Abha, Saudi Arabia} \email{deajim@kku.edu.sa, deajim@gmail.com}


\date{\today}
\maketitle

\begin{center}


{\sc M.E. Charkani }\\
{\footnotesize Department of Mathematics, Faculty of Sciences Dhar-Mahraz, Morocco}\\

{\sc A. Deajim }\\
{\footnotesize Department of Mathematics, King Khalid University, Saudi Arabia}\\
\end{center}

\begin{abstract}
Let $R$ be a Dedekind ring, $\p$ a nonzero prime ideal of $R$, $P\in R[X]$ a monic irreducible polynomial, and $K$ the quotient field of $R$. We give in this paper a lower bound for the $\p$-adic valuation of the index of $P$ over $R$ in terms of the degrees of the monic irreducible factors of the reduction of $P$ modulo $\p$. As an important application, when the lower bound is greater than zero for some $\p$, we conclude that no root of $P$ generates a power integral basis in the field extension of $K$ defined by $P$.
\end{abstract}

\section{\bf Introduction}
It is well-known that the problem of studying the integral closure of a Dedekind ring $A$ in some finite separable extension of its quotient field is related to the problem of studying the integral closures of localizations of $A$ at its nonzero prime ideals. As such localizations are discrete valuation rings, we consider some tools and relevant results over discrete valuation rings.

Let $R$ be a discrete valuation ring and $\p$ its nonzero prime ideal. An important notion which we need here is the notion of index of an irreducible polynomial over $R$, whose definition we recall below for the sake of completion (see \cite{CD2} for a more general treatment).

If $M$ is a nonzero torsion $R$-module of finite type, then $M$ admits a composition series of submodules
$$0=M_0 \subseteq M_1 \subseteq \dots \subseteq M_t =M,$$
where, for each $i=0,1, \dots, t-1$, the quotient $R$-module $M_{i+1}/M_i$ is simple and is isomorphic to $ R/\mathfrak{p}$. Note here that $t$ is an invariant of $M$ independent of the choice of composition series (see \cite[Proposition 6.7]{At-Mac} or \cite[\S 11--Theorem 19]{ZS}). Now, if $M$ is an $R$-module of finite type, we define {\it the order ideal of $M$ over $R$} by (see \cite{CD2}):
$$\mbox{ord}_R(M):= \left \{ \begin{array} {c@{\quad;\quad}l} R & \mbox{if}\, M=0 \\
0 & \mbox{if}\, M \, \mbox{is not a torsion $R$-module}\\
\mathfrak{p}^t & \mbox{if}\, M \, \mbox{is a nonzero torsion $R$-module}.
\end{array} \right. $$
If $N \subseteq M$ are two projective $R$-modules of the same finite constant rank, then we define the {\it index of $M$ over $N$} to be the ideal $\mbox{ord}_R(M/N)$, and we denote it by $[M:N]_R$. By definition, note that $M=N$ if and only if $[M:N]_R=R$.

Let $K$ be the quotient field of $R$, $L$ a finite separable extension of $K$, $O_L$ the integral closure of $R$ in $L$, $\alpha\in O_L$ a primitive element of $L$, and $P \in R[X]$ the minimal polynomial of $\alpha$ over $R$. As $R[\al] \subseteq O_L$ are projective $R$-modules of the same constant rank, $[O_L: R[\alpha]]_R$ is well defined. We call $[O_L:R[\al]]_R$ the {\it the index of $P$} (or {\it of $\alpha$}) and denote it by $\mbox{Ind}_R(P)$ (see \cite{Cha-Dea} for instance). This notion of index is generalized to the case when $R$ is a Dedekind domain (see \cite{CD2}).
Now recall the generalized discriminant-index formula (see \cite{CD2})
$$\mbox{Disc}_R(P)=\mbox{Ind}_R(P)^2\, D_R(O_L),$$
where $\mbox{Disc}_R(P)=\mbox{disc}(P)R$ is the principal ideal of $R$ generated by the usual discriminant of $P$, and $D_R(O_L)$ is the relative discriminant of $O_L$ over $R$ in the sense of \cite{FT} (see also \cite{EC}).

If $R$ is a Dedekind ring, $\p$ a nonzero prime ideal of $R$, and $P(X)\in R[X]$ is a monic irreducible polynomial defining a separable field extension $L$ over the quotient field of $R$, then our main result, Thoerem \ref{main} below, gives a lower bound on the $\p$-adic valuation of $\mbox{Ind}_R(P)$ in terms of the degrees of the monic irreducible factors of the reduction of $P$ modulo $\p$. An important application of this theorem is to conclude, when the lower bound is greater than zero for some $\p$, that a root of $P(X)$ in $L$ does not generate a power basis for the integral closure of $R$ in $L$.

\begin{thm}\label{main}
Let $R$ be a Dedekind ring, $K$ its quotient field, $L$ a finite separable extension of $K$, $O_L$ the integral closure of $R$ in $L$, $L=K(\al)$ for some $\al\in O_L$, $P\in R[X]$ the minimal polynomial of $\al$ over $K$, and $A=R[\al]$. Let $\p$ be a nonzero prime ideal of $R$, $\ol{P}=\prod_{i=0}^r \ol{P_i}^{l_i}$ the monic irreducible factorization of $\ol{P}$ modulo $\p$, and $P_i\in R[X]$ a monic lift of $\ol{P_i}$ for each $i$. Suppose that $T\in R[X]$ with $\ol{T}$ nonzero and $P=\prod_{i=0}^r P_i^{l_i} +a T$ for $a\in \p-\p^2$. Let $t_i$ be the highest power of $\ol{P_i}$ that divides $\ol{T}$ and set $s_i=\mbox{min}\{[\cfrac{l_i}{2}], t_i\}$. Then,
$$\nu_\p(\mbox{Ind}_R(P) \geq \sum_{i=1}^r s_i \mbox{deg}\,(P_i)\;(=\mbox{deg}\,(\prod_{i=1}^r P_i^{s_i})).$$
\end{thm}

\section{\bf Lemmas}
For a nonzero fractional ideal $\mathfrak{a}$ of an integral domain $R$, denote by $\mbox{Stab}_K(\mathfrak{a})$ the fractional ideal
$$(\mathfrak{a}:_K \mathfrak{a}):=\{\, x\in K \, | \, x\mathfrak{a} \subseteq \mathfrak{a}\,\},$$
where $K$ is the quotient field of $R$. We call $\mbox{Stab}_K(\mathfrak{a})$ the $K$-stabilizer of $\mathfrak{a}$. Note that $R\subseteq \mbox{Stab}_K(\mathfrak{a})$. It is known that $\mbox{Stab}_K(\mathfrak{a})$ is the largest subring $B$ of $K$ such that $\mathfrak{a}$ is a $B$-module. Further, if $R$ is noetherian and $\ol{R}$ is the integral closure of $R$ in $K$, then $\ol{R}=\bigcup_{\mathfrak{a}} \mbox{Stab}_K(\mathfrak{a})$, where $\mathfrak{a}$ runs over all nonzero fractional ideals $\mathfrak{a}$ of $R$ (see \cite[Proposition 2.4.8]{SH}).

By $(.\, , .)$ (resp. $[.\, , .]$) we mean the usual notation for the greatest common divisor (resp. least common multiple), and by $[.]$ we mean the integral part of a number.

Now let $R$ be a discrete valuation ring, $\pi$ a uniformizer of $R$, $\p=\pi R$ the prime ideal of $R$, $k=R/\p$ the residue field of $R$, $K$ the quotient field of $R$, $L$ a finite separable extension of $K$, $O_L$ the integral closure of $R$ in $L$, $L=K(\al)$ for some $\al \in O_L$, $P(X)\in R[X]$ the minimal polynomial of $\al$ over $R$, and $A=R[\al]$. For a polynomial $f(X)=\sum_{i=0}^t a_i X^i\in R[X]$, by $\ol{f}(X)$ we mean the polynomial $\sum_{i=0}^t \ol{a_i} X^i \in k[X]$ resulting from reducing all coefficients of $f(X)$ modulo $\p$. \\

\begin{lem}\label{V}
Keep the notation as above. Let $W \in R[X]$ be monic such that
$\ol{W}$ divides $\ol{P}$ in $k[X]$. Then $\,M=\,A + (W(\alpha)/\pi)\,A$ is a free $R$-module with $\displaystyle{[M\,: \,A]_R = \mathfrak{p}^{n-m},}$
where $n = \mbox{deg}\,(P)$ and $m=\mbox{deg}\,(W)$.
\end{lem}

\begin{proof}
Since $R$ is a discrete valuation ring and $A$ is a free $R$-module of finite rank $n$ and $\pi\,M \subseteq A \subseteq M$, it follows that $\pi\,M$ and, thus, $M$ are free $R$-modules of rank $n$ as well. Since $A/\pi\,M \cong k[X]/(\ol{W}) \cong k^m = (R/\mathfrak{p})^m$ as $R$-modules and $[A:\pi\,M]_R =\mbox{ord}_R(A/\pi\,M)$, $[A:\pi\,M]_R =\mbox{ord}_R((R/\mathfrak{p})^m)=\mathfrak{p}^m$. On the other hand, as $[M:A]_R\,[A:\pi\,M]_R=[M:\pi\,M]_R=\mbox{ord}(M/\pi\,M)=\mathfrak{p}^n,$ it follows that $[M:A]_R=\mathfrak{p}^{n-m}$ as claimed.
\end{proof}

\begin{lem}\label{Stab}
Keep the notation of Lemma \ref{V}. Let $\displaystyle{\mathfrak{a}=\pi \,A+f(\alpha)\,A}$ for some monic $f\in R[X]$ such that $\overline{f}$ divides $\overline{P}$ in $k[X]$. Let $g, T \in R[X]$ be such that $P=fg +\pi T$ with $\overline{T}$ nonzero, and $D, h, U\in R[X]$ such that $\overline{D}=(\overline{f}, \overline{g}, \overline{T})$,  $\overline{h}=\overline{f}/(\overline{f}, \overline{T})$, and $\ol{U}=\ol{P}/\ol{D}$. Then we have the following:

(i) $\ol{U}=[\overline{h}\overline{g}, \overline{f}]$.

(ii) $\mbox{Stab}_L(\mathfrak{a})= A + \cfrac{U(\alpha)}{\pi}\, A \subseteq O_L$.

(iii) $\mbox{Stab}_L(\mathfrak{a})$ is a free $R$-submodule of $O_L$ and $[\mbox{Stab}_L(\mathfrak{a}):\,A]_R = \mathfrak{p}^{\mbox{deg}\,(D)}$.
\end{lem}

\begin{proof}$\\$
\indent (i) $
\ol{U}=\cfrac{\ol{g}\,\ol{f}}{(\ol{g},
\ol{f}, \ol{T})}=\cfrac{\ol{g}\,\ol{h}(\ol{f},\ol{T})}{(\ol{g},
(\ol{f}, \ol{T}))}= \ol{h}[\ol{g}, (\ol{f}, \ol{T})] = [\ol{h}\,
\ol{g}, \ol{h} (\ol{f}, \ol{T})] = [\ol{h}\,\ol{g}, \ol{f}].$

\indent (ii) Note, first, that $\mathfrak{a}$ is a nonzero fractional ideal of $A$. So, $\mbox{Stab}_L(\mathfrak{a}) \subseteq O_L$ follows from the fact that $O_L=\bigcup_{\mathfrak{a}} \mbox{Stab}_L(\mathfrak{a})$, where $\mathfrak{a}$ runs over all nonzero fractional ideals of $A$ (see the paragraph preceding this lemma). It remains to show the proposed equality. By part (i),
$
\ol{U}= [\ol{h}\,\ol{g}, \ol{f}]$. So, it suffices to show that if $x=Q(\alpha)/\pi$ for some $Q\in R[X]$, then $x\in \mbox{Stab}_L(\mathfrak{a})$ if and only if both $\ol{f}$ and $\ol{g}\,\ol{h}$ divide $\ol{Q}$. Indeed, as $x\in \mbox{Stab}_L(\mathfrak{a})$ if and only if $\pi x, xf(\alpha) \in \mathfrak{a}$, proving the equality will be complete by showing that $\pi x \in \mathfrak{a}$ if and only if $\ol{f}$ divides $\ol{Q}$, and $xf(\alpha) \in \mathfrak{a}$ if and only if $\ol{g}\,\ol{h}$ divides $\ol{Q}$.

On the one hand, we see that $\pi x\in \mathfrak{a}$ if and only if
there exist $F, G\ \in  R[X]$ such that $Q(\alpha) = \pi F(\alpha) +
f(\alpha) G(\alpha)$. Since $P$ is the minimal polynomial of
$\alpha$ over $R$, this is equivalent to the existence of $H \in
R[X]$ such that $$Q = \pi F + f\,G + P\, H.$$ Substituting
appropriately yields $$Q = \pi(F + T\, H) + f\, (G + g\, H).$$ Thus
$\pi x\in \mathfrak{a}$ if and only if $\ol{f}$ divides $\ol{Q}$.

On the other hand, we see that $x f(\alpha) \in \mathfrak{a}$ if and
only if there exist $Q_1, Q_2 \in R[X]$ such that
\begin{equation}\label{Q1Q2}
Q(\alpha) f(\alpha) = \pi (\pi Q_1(\alpha) + f(\alpha) Q_2(\alpha)).
\end{equation} Since $P$ is the minimal polynomial of $\alpha$ over
$R$, (\ref{Q1Q2}) is equivalent to the existence of $Q_3\in R[X]$
such that
\begin{equation}\label{Q3}
Q\, f = \pi ( \pi Q_1 + f\, Q_2) +P \, Q_3.
\end{equation}
Reducing modulo $\mathfrak{p}$ we get that $\ol{Q} = \ol{Q_3}\,
\ol{g}$. Let now $Q_4 \in R[X]$ be such that
\begin{equation}\label{Q4}
Q = \pi Q_4 + Q_3\, g.
\end{equation}
Substituting in (\ref{Q3}), we get $$ Q_3\, (f\, g - P) = \pi (\pi
Q_1 + f\, Q_2 - f\, Q_4).$$ Now letting $Q_5= Q_4 - Q_2 \in R[X]$,
we get that $xf(\alpha) \in \mathfrak{a}$ if and only if there exist
$Q_1, Q_3, Q_5 \in R[X]$ such that $$Q_3\, T = f\, Q_5 - \pi Q_1.$$
Reducing modulo $\mathfrak{p}$ again, we get $\ol{Q_3} \, \ol{T} =
\ol{f} \, \ol{Q_5}$. This is equivalent to saying that $\ol{f}$
divides $\ol{T}\, \ol{Q_3}$ or, equivalently, $\ol{h}$ divides
$\ol{Q_3}$ since $\ol{T}$ is nonzero. So $xf(\alpha) \in \mathfrak{a}$ if and only if there
exist $Q_6, Q_7 \in R[X]$ such that $$Q_3 =h\, Q_6 + \pi Q_7.$$
Substituting in (\ref{Q4}) we get
\begin{align*}
Q &= \pi (Q_4 +g\, Q_7 ) + g\, h\, Q_6\\ &= \pi Q_8 + g\, h\, Q_6,
\end{align*}
where $Q_8 = Q_4 + g Q_7 \in R[X]$. It finally follows that
$xf(\alpha) \in \mathfrak{a}$ if and only if $\ol{g}\, \ol{h}$
divides $\ol{Q}$.

(iii) Apply Lemma \ref{V} with $U = W$.
\end{proof}

\noindent{\bf Remark.}
It should be noted here that Lemma \ref{Stab} generalizes Lemma 6.1.5 of \cite{Cohen}. Furthermore, choosing $\ol{f}$ to be the radical of the square part of $\ol{P}$ would give a refinement of Theorem 6.1.5 of \cite{Cohen} (by the radical of a polynomial over $K$ we mean the product of all its distinct irreducible factors, and by the square part of a polynomial over $K$ we mean the quotient upon dividing the polynomial by its radical).

\section{\bf Lower Bounds of $\nu_\p(\mbox{Ind}_R(P))$}

\begin{prop}\label{IndexDegree}
Keep the notation and assumptions of Lemma \ref{Stab}. Then
$$\nu_\pi(\mbox{Ind}_R(P)) \geq \mbox{deg}\,(D).$$
\end{prop}

\begin{proof}
Indeed, $A \subseteq \mbox{Stab}_L(\mathfrak{a}) \subseteq O_L$ and, therefore, $[\mbox{Stab}_L(\mathfrak{a}) :\,A]_R$ divides \linebreak $[O_L: \,A]_R =\mbox{Ind}_R(P)$. Now, using Lemma \ref{Stab} (ii) yields the claim.
\end{proof}

\begin{prop}
Keep the notation and assumptions of Lemma \ref{Stab}. If $O_L=A$, then $(\ol{f}, \ol{h}, \ol{T})=\ol{1}$.
\end{prop}

\begin{proof}
It follows from Lemma \ref{Stab} that $\mbox{Stab}_L(\mathfrak{a})=O_L$ and $\mbox{deg}(D)=0$.
\end{proof}

Now the proof of Theorem \ref{main}

\begin{proof} ({\bf Theorem \ref{main}})
We localize at $\p$ and apply Proposition \ref{IndexDegree} to a suitable choice of a
divisor $\ol{f}$ of $\ol{P}$; namely, $\ol{f}=\prod_{i=1}^r \ol{P_i}^{s_i}$. Let $f=\prod_{i=1}^r\, P_i^{s_i}$ and $h=\prod_{i=1}^r\, P_i^{l_i-s_i}$. Then, $P=fh+\pi T$ and,
in this case, $\ol{D}=(\ol{f}, \ol{h}, \ol{T}) = \prod_{i=1}^r
\ol{P_i}\,^{m_i}$, where $m_i=\mbox{min}\{s_i, l_i - s_i, t_i\}$.
Since $m_i = \mbox{min}\{[\cfrac
{l_i}{2}], t_i\} = s_i$, the assertion now follows from Proposition \ref{IndexDegree}.
\end{proof}

\noindent{\bf Example.}

Let $R$ be a Dedekind ring, $K$ the quotient field of $R$, $\p=\pi R$ a nonzero principal prime ideal of $R$, and $P(X)=g(X)^n + a g(X)^m + b\in R[X]$ irreducible over $R$, with $g(X)\in R[X]$ monic, $\ol{g}(X)$ irreducible modulo $\p$, $n >1$, $n\geq m\geq 1$, $\nu_\p(a)=1$, and $\nu_\p(b)\geq 2$ (e.g. $P(X)=X^3+\sqrt{3} X^2 + 3 \in \Z[\sqrt{3}][X]$ with $\p=\sqrt{3}\Z[\sqrt{3}]$). Then, $\ol{P}(X) \equiv \ol{g}(X)^n$ modulo $\p$. For $T(X)=g(X)^m + b/\pi \in R[X]$, it is clear that $\ol{T}(X)=\ol{g}(X)^m$ is nonzero and $\ol{T}(X)$ divides $\ol{P}(X)$ modulo $\p$. As $l=n > 1$ and $t=m \geq 1$, $s=\mbox{min}\{[\cfrac{n}{2}], m\} \geq 1$. Thus, $\nu_\p(\mbox{Ind}_R(P)) \geq (1)\mbox{deg}(g(X))\geq 1$. Hence, a root $\al$ of $P$ never generates a power basis for the integral closure of $R$ in $K(\al)$.\\

\noindent{\bf Remark.} It was shown in \cite[Theorem 1.1 and Corollary 1.2]{JKS} that if $P(X)=X^n+aX^m+b\in \Z[X]$ is irreducible with $m|n$, $p|a$, $p^2|b$ for a prime integer $p$, then $p |\mbox{Ind}_\Z(P)$. This is a special case that follows from the example above without even requiring that $m|n$.


\end{document}